\documentclass[11pt]{article}
\date{}
\usepackage{amssymb}
\usepackage{amsmath}
\usepackage{amsthm}
\usepackage{graphicx}
\usepackage{indentfirst}
\allowdisplaybreaks[4]
\newtheorem{theorem}{Theorem}

\newtheorem{corollary}[theorem]{Corollary}

\newtheorem{definition}[theorem]{Definition}

\newtheorem{lemma}[theorem]{Lemma}

\newtheorem{remark}[theorem]{Remark}

\numberwithin{equation}{section}
\numberwithin{theorem}{section}
\newcommand{\keywords}[1]{\par\noindent\textbf{Keywords:} #1}
\newcommand{\subjclass}[2]{
  \par\noindent\textbf{Mathematics Subject Classification} #2}

\begin{document}

\title{On nonnegative solutions of the differential inequality $\Delta_pu+ \Delta_q u+V(x)u^s\leq 0$ on Riemannian manifolds }

\author{Biqiang Zhao\\
\small Beijing International Center for Mathematical Research,
\\
\small Peking University, Beijing, China\\ \small 2306394354@pku.edu.cn}
            
\maketitle
\setlength{\parindent}{2em}

\begin{abstract}
In this paper, we are concerned with differential inequalities
with $(p,q)$-Laplacian operator on Riemannian manifolds. Using a test function argument, we establish Liouville-type theorems under the manifold's geometry  and the potential's behavior at infinity. Finally, we give an example showing the sharpness when $V(x)$ is a positive constant.

\end{abstract}

\subjclass [{ 35J60, 35J92, 53C20, 53A55 }


\section{Introduction}
\label{1}
In recent years, the study of nonexistence results for nonnegative solutions to quasilinear elliptic equations or inequalities have many achievements, cf.\cite{AS,BBF,MMP1,MMP2,MMP3}. The purpose of the present paper is to study the Liouville type theorems for nonnegative solutions of the quasilinear elliptic inequalities
\begin{align*}
    \Delta_pu+ \Delta_q u+V(x)u^s\leq 0, \quad in \ M
\end{align*}
where $(M,g)$ is a complete, non-compact, Riemannian manifold with metric $g$ and $\Delta_z u=\mathrm{div}(|\nabla u|^{z-2}\nabla u),\ 1<z\in \{p,q\}$.
\par
Recalling that in the Euclidean space $\mathbb{R}^n$, due to the work of Gidas and Spruck \cite{GS}, one of the outstanding result is attached to the equation
 \begin{align}
     \Delta u+u^p=0,\quad \mathbb{R}^n,\quad  n>2,\quad p>1.
 \end{align}
 Gidas and Spruck proved that (1.1) admits no nontrivial non-negative solution if $1<p<\frac{n+2}{n-2}$. Moreover, (1.1) admits no nontrivial non-negative supersolution if $1<p\leq\frac{n}{n-2}$. In the case of $p-$Laplacian, Serrin and Zhou \cite{SZ} proved the classical Lane-Emden-Fowler equation
 \begin{align*}
     \Delta_pu+u^s=0 \quad in\ \mathbb{R}^n
 \end{align*}
 admits no positive solution if 
 \begin{align*}
     1<s< \frac{(n+1)p-n}{n-p},\quad 1<p<n.
 \end{align*}
 For supersolutions, Mitidieri and Pohozaev \cite{MP1} obtained the Liouville property for the following inequality
 \begin{align}
     \Delta_p u+u^s\leq 0,\quad in\ \mathbb{R}^n,
 \end{align}
 that is, (1.2) admits no nontrivial non-negative solution if
\begin{align*}
    1<s\leq \frac{(p-1)n}{n-p},\quad 1<p<n.
\end{align*}
For exterior domains in $ \mathbb{R}^n$, Bidaut-V$\acute{e}$ron and Pohozaev \cite{BMP} proved that the nonneagtive solution of (1.2) is $u\equiv 0$ if $ p<n$ and $ 1<s\leq \frac{(p-1)n}{n-p}$ or $ p=n$ and $1<s<\infty$. In \cite{MP2}, Mitidieri and Pohozaev investigated the inequality involving a gradient nonlinearity 
\begin{align*}
    \Delta_p u+u^s|\nabla u|^m\leq 0\quad in \ \mathbb{R}^n,
\end{align*}
where $s>0,m\geq 0 $. They obtained Liouville type theorems when 
\begin{align*}
    s(n-p)+m(n-1)<n(p-1), \quad s+m>p-1.
\end{align*}
The cases $s(n-p)+m(n-1)=n(p-1) $ and $ 0<s\leq p-m-1$ were proved by Filippucci \cite{Fi1,Fi2}. For more Liouville properties, one can refer to \cite{BMGV,BGQ,CM,CHZ}.
\par
Next, let's return to results on Riemannian manifolds. Cheng and Yau \cite{CY} proved that if $V (B_r(x_0))\leq Cr^2$ for all large enough $r$, then any positive solution of $\Delta u\leq 0$ is constant. This idea was developed by Grigor'yan and Kondratiev \cite{GK}. They proved that, for some $C>0,\epsilon>0,x_0\in M$ and large enough $R>0$, the only nonnegative solution to the inequality
\begin{align}
    \Delta u+u^s\leq 0,\quad in\ M
\end{align}
is equal to 0 if 
\begin{align*}
    V(B_R(x_0))\leq C R^{\frac{2s}{s-1}}(\mathrm{log}\  R)^{\frac{1}{s-1}-\epsilon}.
\end{align*}
Here $B_R(x_0) $ is the geodesic ball on $M$ of radius $R$ centered at $x_0$. In fact, Grigor'yan and Kondratiev investigated the differential inequality with a potential. The sharpness of the exponent of $\mathrm{log}\ R$ was solved by Grigor'yan and Sun \cite{AS}, i.e $\epsilon=0$. Later, Mastrolia, Monticelli and Punzo \cite{MMP1} investigated the differential inequality with a potential to the $p-$Laplacian and generalized the result of Grigor'yan and Kondratiev. In 2022, Sun, Xiao and Xu \cite{SXX} proved Liouville type theorems for the quasilinear elliptic inequality 
\begin{align}
    \Delta_p u+u^s|\nabla u|^q\leq 0
\end{align}
on geodesically complete noncompact Riemannian manifolds.
\par
In the last few years, the $(p,q)$-Laplace operator attracts a lot of attention. The idea of studying such operators comes from the problems of the calculus of variations and nonlinear elasticity theory, cf. \cite{Ma1,Ma2,Zh1,Zh2}. Recently, Bhakta, Biswas and Filippucci \cite{BBF} established several Liouville-type theorems for differential quasilinear inequalities with $(p,q)$-Lapalce in the entire $\mathbb{R}^n$ (or an exterior domain). Wang and Zhang \cite{WZ} derived Liouville type theorems for $(p,q)$-harmonic functions on complete noncompact Riemannian manifold with nonnegative Ricci curvature. Bobkov and Tanaka \cite{BT} studied the existence and non-existence of positive solutions for the $(p,q)-$Laplace equation with two parameters.
\par
The aim of this paper is to obtain Liouville-type theorems for quasilinear inequalities driven by the $(p,q)$-Laplace operator on Riemannian manifolds.
\par
Throughout the paper, we assume that $1<q\leq p$ and $r(x)$ is the Riamannian distance from $x$ to a fixed point $x_0\in M$. Denote by $B_R$ the ball centered at $x_0$ with radius $R$. Since the constant $C>0$ is not important, it may vary at different occurrences.
\par
Before starting with the main theorems, we give the following definition. 

\begin{definition}
    Let $p>q>1, s>p-1$, $V>0$ a.e. on $M$ and $V\in L_{loc}^1(M)$. For $z\in \{p,q\}$, we define
    \begin{align*}
        s_z=\frac{zs}{s-z+1}, \quad \bar{k}_z=\frac{z-1}{s-z+1}.
    \end{align*}
    We introduce the following growth conditions HP1, HP2 and HP3:
    \par
    HP1: Assume that, for $z\in \{p,q\}$, there exists $C_0,C>0$ and $k_z\in [0,\bar{k}_z)$ such that for every $R$ large enough and every $\epsilon$ small enough
    \begin{align}
        \int_{B_R\setminus  B_{\frac{R}{2}}} V^{-\bar{k}_z+\epsilon}d\mu \leq CR^{s_z+C_0\epsilon} (\mathrm{log}\ R)^{k_z}.
    \end{align}
    \par
    HP2: Assume that, for $z\in \{p,q\}$, there exists $C_0,C>0$ such that for every $R$ large enough and every $\epsilon$ small enough
    \begin{align}
        &\int_{B_R\setminus B_{\frac{R}{2}}} V^{-\bar{k}_z+\epsilon}d\mu \leq CR^{s_z+C_0\epsilon} (\mathrm{log}\ R)^{\bar{k}_z},
        \\
        &\int_{B_R\setminus B_{\frac{R}{2}}} V^{-\bar{k}_z-\epsilon}d\mu \leq CR^{s_z+C_0\epsilon} (\mathrm{log}\ R)^{\bar{k}_z}.
    \end{align}
    \par
    HP3: Assume that, for $z\in \{p,q\}$, there exists $C_0,C>0$, $k\geq 0$, $ \theta>0$ and $\tau_z >\mathrm{max}\{\frac{s-z+1}{s}(k+1),1\}$ such that for every $R$ large enough and every $\epsilon$ small enough
    \begin{align}
        \int_{B_R\setminus B_{\frac{R}{2}}} V^{-\bar{k}_z+\epsilon}d\mu \leq CR^{s_z+C_0\epsilon} (\mathrm{log}\ R)^{k} e^{-\epsilon\theta(\mathrm{log}\ R)^{\tau_z}}.
    \end{align}
\end{definition}
\begin{remark}
   (1) By Fatou’s Lemma, the above conditions hold for $\epsilon=0$.
   \\
   (2) In general, conditions HP1, HP2 and HP3 are mutually independent (cf. \cite{MMP1}).
\end{remark}
\begin{remark}
    The conditions (1.5)-(1.8) hold when the potential $V$ satisfies the following growth conditions.
    \\
    (1) Let $C_0,k_z$ be defined as in HP1. For $z\in\{p,q\}$, (1.5) holds if there exists constant $C>0$ such that 
    \begin{align}
        V(x)\leq C(1+r(x))^{C_0},\quad \int_{B_R\setminus  B_{\frac{R}{2}}} V^{-\bar{k}_z}d\mu \leq CR^{s_z} (\mathrm{log}\ R)^{k_z}
    \end{align}
    for $R>0$ large enough.
    \\
    (2) Let $C_0$ be defined as in HP2. For $z\in\{p,q\}$, (1.6) and (1.7) hold if there exists constant $C>0$ such that
    \begin{align}
        C^{-1}(1+r(x))^{-C_0}\leq V(x)\leq C(1+r(x))^{C_0},\quad \int_{B_R\setminus  B_{\frac{R}{2}}} V^{-\bar{k}_z}d\mu \leq CR^{s_z} (\mathrm{log}\ R)^{\bar{k}_z}
    \end{align}
    for $R>0$ large enough.
    \\
    (3) Let $C_0,k,\theta,\tau_z$ be defined as in HP3. For $z\in\{p,q\}$, (1.8) holds if there exists constant $C>0$ such that 
    \begin{align}
        V(x)\leq C(1+r(x))^{C_0}e^{-\theta(\mathrm{log}\ r(x))^{\tau_z}},\quad \int_{B_R\setminus  B_{\frac{R}{2}}} V^{-\bar{k}_z}d\mu \leq CR^{s_z} (\mathrm{log}\ R)^{k}
    \end{align}
    for $R>0$ large enough.
\end{remark}
     \begin{remark}
         The conditions highlight the dominant influence of the lower order term, that is $q$-Laplacian operator, which primarily affects the structure of the analysis. For example, if $V(x)=1$, then we only need the conditions hold for $z=q$.
         
     \end{remark}
       Now we state the main theorem of this paper (the notion of weak solutions is in Section 2).
       \begin{theorem}
           Let $p>q>1,\ s>p-1$ and $ V(x)\in L_{loc}^1(M)$ with $V> 0$ a.e. on $M$. Assume that the condition HP1 holds. Let $u\in W^{1,p}_{loc}(M)\cap L^s_{loc}(M)$ be a nonnegative weak solution of 
           \begin{align}
               \Delta_pu+ \Delta_q u+V(x)u^s\leq 0, \quad in \ M
           \end{align}
           then $u\equiv0$ a.e on $M$.
       \end{theorem}
       \begin{theorem}
           Let $p>q>1,\ s>p-1$ and $ V(x)\in L_{loc}^1(M)$ with $V> 0$ a.e. on $M$. Assume that the condition HP2 holds. Let $u\in W^{1,p}_{loc}(M)\cap L^s_{loc}(M)$ be a nonnegative weak solution of (1.12),
           then $u\equiv0$ a.e on $M$.
       \end{theorem}
       \begin{theorem}
           Let $p>q>1,\ s>p-1$ and $ V(x)\in L_{loc}^1(M)$ with $V> 0$ a.e. on $M$. Assume that the condition HP3 holds. Let $u\in W^{1,p}_{loc}(M)\cap L^s_{loc}(M)$ be a nonnegative weak solution of (1.12),
           then $u\equiv0$ a.e on $M$.
       \end{theorem}
       \par
        Theorem 1.5, Theorem 1.6 and Theorem 1.7 can be regard as the $(p,q)- $Laplace counterparts of the results in \cite{MMP1}. From Remark 1.3, we have the following corollary.
       \begin{corollary}
           Let $p,q,V $ and $u$ be defined as in Theorem 1.5. Then $u\equiv0$ a.e on $M$ provided that one of the (1.9), (1.10) or (1.11) holds.
       \end{corollary}
       For the Euclidean space $\mathbb{R}^n$, we have 
       \begin{corollary}
           Let $(M,g)=(\mathbb{R}^n,g_{flat}),p>q>1$ and $V=1$. If 
           \begin{align*}
               p-1<s, \quad (n-q)s\leq n(q-1),
           \end{align*}
           then any nonnegative weak solution of (1.12) is equal to 0 a.e on $M$.
       \end{corollary}
           Corollary 1.9 improves the result in \cite{BBF} by proving the nonexistence results for the range $q\geq n$. Similar results can be obtained on Riemannian manifolds with nonnegative Ricci curvature. 
           \begin{remark}
               From the proof of the Theorem, we can establish nonexistence results for a wider class of quasilinear inequalities of the type
               \begin{align*}
                   \mathrm{div}(|\nabla u|^{q-2}f(|\nabla u|)\nabla u)+V(x,t)u^s\leq 0,\quad in\  M,
               \end{align*}
               where $f$ satisfies the condition that there exists constants $b\geq a>0,K\geq 0$ and $v> 0$ such that
               \begin{align*}
                   at^{v}\leq f(t) \leq K+bt^v, \quad \forall t\geq 0.
               \end{align*}
           \end{remark} 
           Finally, we give an example showing the sharpness when $V(x)=K>0$ is a positive constant. In this case, HP2 is equal to 
           \begin{align}
               Vol(B_x(r)) \leq Cr^{\frac{sq}{s-q+1}}(\mathrm{ln}\ r)^{\frac{q-1}{s-q+1}} \quad \forall R>>1,
           \end{align}
           where $Vol(B_x(r)) $ is the volume of a Riemannian ball $B_x(r)$ centered at $x\in M$ with radius $r$. We show that, for $\beta>\frac{q-1}{s-q+1} $, there exists a manifold $(M,g)$ satisfying 
           \begin{align}
               Vol(B_x(r)) \leq Cr^{\frac{sq}{s-q+1}}(\mathrm{ln}\ r)^{\beta} \quad \forall r>>1,
           \end{align}
           where the inequality 
           \begin{align}
               \Delta_pu+ \Delta_q u+K u^s\leq 0
           \end{align}
           has a nontrivial positive solution. 
           \par
           The rest of the paper is organized as follows. In Section 2, we prove some preliminary results, which will be used in the proof. In Section 3 to 5, we give the proof of Theorem 1.5, Theorem 1.6 and Theorem 1.7 respectively. In section 6, we give the example when $V(x)$ is a positive constant.

           \section{Preliminary}
           First we give the definition of the weak solution of (1.12). 
           \begin{definition}
               Let $p>q>1,\ s>p-1$ and $ V(x)\in L_{loc}^1(M)$ with $V> 0$ a.e.on $M$. We say that $u\in W^{1,p}_{loc}(M)\cap L^s_{loc}(M)$ is a nonnegative weak solution of (1.12) if $u\geq 0$ and for every $0\leq \psi\in W^{1,p}(M)\cap L^{\infty}(M)$ with compact support, one has
               \begin{align}
                   -\int_M |\nabla u|^{p-2}\langle\nabla u,\nabla\psi\rangle d\mu -\int_M |\nabla u|^{q-2}\langle\nabla u,\nabla\psi\rangle d\mu+\int_MVu^s\psi d\mu \leq 0.
               \end{align}
           \end{definition}
           Using a test function argument, we give the following two lemmas.
           \begin{lemma}
               Assume that $u$ is a nonnegative weak solution of (1.12), then there exist a positive pair $(a,b)$ and a constant $C>0$ such that for any $ \varphi\in \mathrm{Lip}(M)$ with compact support and $0\leq \varphi\leq 1$, the following estimate holds
               \begin{align}
                  & \frac{a}{2}\int_M |\nabla u|^pu^{-1-a}\varphi^b d\mu+\frac{a}{2}\int_M |\nabla u|^q u^{-1-a}\varphi^b d\mu+\frac{1}{2}\int_M Vu^{s-a}\varphi^b d\mu
                    \nonumber\\
                   \leq &C a^{-\frac{s(p-1)}{s-p+1}}\int_M V^{-\frac{p-a-1}{s-p+1}} |\nabla \varphi|^{\frac{(s-a)p}{s-p+1}}d\mu+C a^{-\frac{s(q-1)}{s-q+1}}\int_M V^{-\frac{q-a-1}{s-q+1}} |\nabla \varphi|^{\frac{(s-a)q}{s-q+1}}d\mu   ,        
               \end{align}
               where $ a\in (0,\mathrm{min}\{1,q-1\})$ and $b>\frac{ps}{s-p+1}$.
           \end{lemma}
           \begin{proof}
               Without loss of generality, we assume that $u$ is positive with 
               $u^{-1}\in L^{\infty}_{loc}(M)$. Otherwise, we define $u_\eta=u+\eta$ for any $\eta>0$ and then let $\eta\xrightarrow{} 0$.  Let $\psi=u^{-a}\varphi^b$, then
               \begin{align}
                   \nabla\psi=-au^{-a-1}\varphi^b\nabla u+bu^{-a}\varphi^{b-1}\nabla \varphi\quad a.e.\ on\ M.
               \end{align}
               Combining with (2.1), we have 
               \begin{align}
                   &a\int_M |\nabla u|^pu^{-1-a}\varphi^b d\mu+a\int_M |\nabla u|^q u^{-1-a}\varphi^b d\mu+\int_M Vu^{s-a}\varphi^b d\mu\nonumber
                   \\
                   \leq &b\int_M|\nabla u|^{p-2}u^{-a}\varphi^b\langle \nabla u,\nabla \varphi\rangle d\mu+b\int_M|\nabla u|^{q-2}u^{-a}\varphi^b\langle \nabla u,\nabla \varphi\rangle d\mu.
               \end{align}
               By the Young's inequality with the pair $(\frac{p-1}{p},\frac{1}{p})$, i.e.,
               \begin{align*}
                   XY\leq X^{\frac{p}{p-1}}+Y^{p}, \quad \forall X,Y>0,
               \end{align*}
               we obtain
               \begin{align}
                   &b\int_M|\nabla u|^{p-2}u^{-a}\varphi^b\langle \nabla u,\nabla \varphi\rangle d\mu\leq b\int_M|\nabla u|^{p-1}u^{-a}\varphi^b|\nabla \varphi| d\mu \nonumber
                   \\
                   = &\int_M\left[|\nabla u|^{p-1}\left(\frac{a}{2}\right)^{\frac{p-1}{p}}\varphi^{\frac{p-1}{p}}u^{-(a+1)\frac{p-1}{p}} \right] \cdot \left[b \left(\frac{a}{2}\right)^{-\frac{p-1}{p}}\varphi^{\frac{b}{p}-1}u^{1-\frac{a+1}{p}} |\nabla \varphi| \right]d\mu \nonumber
                   \\
                   \leq&  \frac{a}{2}\int_M |\nabla u|^pu^{-1-a}\varphi^b d\mu+b^p\left(\frac{a}{2}\right)^{-(p-1)}\int_M\varphi^{b-p}u^{p-a-1}|\nabla \varphi|^pd\mu.
               \end{align}
               Similarly, we have
               \begin{align}
                  & b\int_M|\nabla u|^{q-2}u^{-a}\varphi^b\langle \nabla u,\nabla \varphi\rangle d\mu
                  \nonumber\\
                  \leq &\frac{a}{2}\int_M |\nabla u|^qu^{-1-a}\varphi^b d\mu+b^q\left(\frac{a}{2}\right)^{-(q-1)}\int_M\varphi^{b-q}u^{q-a-1}|\nabla \varphi|^qd\mu.
               \end{align}
               Substituting (2.5) and (2.6) into (2.4), we have
               \begin{align}
                  & \frac{a}{2}\int_M |\nabla u|^pu^{-1-a}\varphi^b d\mu+\frac{a}{2}\int_M |\nabla u|^q u^{-1-a}\varphi^b d\mu+\int_M Vu^{s-a}\varphi^b d\mu
                    \nonumber\\
                   \leq &b^p\left(\frac{a}{2}\right)^{-(p-1)}\int_M\varphi^{b-p}u^{p-a-1}|\nabla \varphi|^pd\mu   +   b^q\left(\frac{a}{2}\right)^{-(q-1)}\int_M\varphi^{b-q}u^{q-a-1}|\nabla \varphi|^qd\mu  .
               \end{align}
               Applying Young's inequality with the pair $(\frac{p-a-1}{s-a},\frac{s-p+1}{s-a})$, we obtain
               \begin{align}
                   &b^p\left(\frac{a}{2}\right)^{-(p-1)}\int_M\varphi^{b-p}u^{p-a-1}|\nabla \varphi|^pd\mu
                   \nonumber\\
                   \leq& \int_M\left[u^{p-a-1}\left(\frac{1}{4}\right)^{\frac{p-a-1}{s-a}}V^{\frac{p-a-1}{s-a}}\varphi^{b\frac{p-a-1}{s-a}}  \right] 
                   \nonumber\\
                   &\cdot \left[\left(\frac{1}{4}\right)^{-\frac{p-a-1}{s-a}}b^p\left(\frac{a}{2}\right)^{-(p-1)}\varphi^{b\frac{s-p+1}{s-a}-p}V^{-\frac{p-a-1}{s-a}}|\nabla\varphi|^p  \right] d\mu
                   \nonumber\\
                   \leq &\frac{1}{4}\int_M Vu^{s-a}\varphi^bd\mu+Ca^{-(p-1)\frac{s-a}{s-p+1}}\int_MV^{-\frac{p-a-1}{s-p+1}}|\nabla \varphi|^{p\frac{s-a}{s-p+1}}\varphi^{b-\frac{(s-a)p}{s-p+1}}d\mu 
                   \nonumber\\
                   \leq &\frac{1}{4}\int_M Vu^{s-a}\varphi^bd\mu+Ca^{-(p-1)\frac{s-a}{s-p+1}}\int_MV^{-\frac{p-a-1}{s-p+1}}|\nabla \varphi|^{p\frac{s-a}{s-p+1}}d\mu.
               \end{align}
               The last inequality holds since $b>\frac{ps}{s-p+1}$ and $0\leq \varphi\leq 1$. In a similar way, we have
               \begin{align}
                   &b^q\left(\frac{a}{2}\right)^{-(q-1)}\int_M\varphi^{b-q}u^{q-a-1}|\nabla \varphi|^pd\mu
                   \nonumber\\
                   \leq &\frac{1}{4}\int_M Vu^{s-a}\varphi^bd\mu+Ca^{-(q-1)\frac{s-a}{s-q+1}}\int_MV^{-\frac{q-a-1}{s-q+1}}|\nabla \varphi|^{q\frac{s-a}{s-q+1}}d\mu.
               \end{align}
              Substituting (2.8) and (2.9) into (2.7), we obtain (2.2).
           \end{proof}
           \begin{lemma}
               Let $b>\frac{2ps}{s-p+1}$ and $a\in (0,\mathrm{min}\{1,q-1,\frac{s-p+1}{2(p-1)}\})$. Assume that $u$ is a nonnegative weak solution of (1.12), then there exists a constant $C>0$ such that for every function $\varphi\in Lip(M)$ with compact support and $0\leq \varphi\leq 1$, we have
               \begin{align}
                   \int_MVu^s\varphi^bd\mu
                   \leq& C(a^{-1}Q)^{\frac{p-1}{p}}\cdot \left(\int_{M\setminus K}Vu^s\varphi^bd\mu \right)^{\frac{(a+1)(p-1)}{sp}}
                   \nonumber\\
                   &\cdot\left(\int_{M\setminus K}V^{-\frac{(a+1)(p-1)}{s-(p-1)(a+1)}}|\nabla\varphi|^{\frac{ps}{s-(a+1)(p-1)}} d\mu\right)^{\frac{s-(a+1)(p-1)}{sp}}
                   \nonumber\\
                    &+C(a^{-1}Q)^{\frac{q-1}{q}}\cdot \left(\int_{M\setminus K}Vu^s\varphi^bd\mu \right)^{\frac{(a+1)(q-1)}{sq}}
                   \nonumber\\
                   &\cdot\left(\int_{M\setminus K}V^{-\frac{(a+1)(q-1)}{s-(q-1)(a+1)}}|\nabla\varphi|^{\frac{qs}{s-(a+1)(q-1)}} d\mu\right)^{\frac{s-(a+1)(q-1)}{sq}},
               \end{align}
               where $K=\{x\in M:\varphi(x)=1\}$ and 
               \begin{align*}
                   Q=a^{-\frac{s(p-1)}{s-p+1}}\int_M V^{-\frac{p-a-1}{s-p+1}} |\nabla \varphi|^{\frac{(s-a)p}{s-p+1}}d\mu+ a^{-\frac{s(q-1)}{s-q+1}}\int_M V^{-\frac{q-a-1}{s-q+1}} |\nabla \varphi|^{\frac{(s-a)q}{s-q+1}}d\mu.
               \end{align*}
           \end{lemma}
           \begin{proof}
               Without loss of generality, we assume that $u$ is positive with 
               $u^{-1}\in L^{\infty}_{loc}(M)$. Taking $\psi=\varphi^b$ in (2.1), we have
               \begin{align}
                   \int_MVu^s\varphi ^b d\mu \leq b\int_M|\nabla u|^{p-2}\varphi^{b-1}\langle \nabla u,\nabla \varphi\rangle d\mu +b\int_M|\nabla u|^{q-2}\varphi^{b-1}\langle \nabla u,\nabla \varphi\rangle d\mu.
               \end{align}
               Using H$\ddot{\mathrm{o}}$lder inequality, we obtain
               \begin{align}
                   &b\int_M|\nabla u|^{p-2}\varphi^{b-1}\langle \nabla u,\nabla \varphi\rangle d\mu
                   \nonumber\\
                   \leq& b\int_M|\nabla u|^{p-1}\varphi^{b-1}|\nabla \varphi|d\mu 
                   \nonumber\\
                   =&b\int_M\left(|\nabla u|^{p-1}\varphi^{b\frac{p-1}{p}}u^{-(a+1)\frac{p-1}{p}} \right)\cdot \left( \varphi^{\frac{b}{p}-1}u^{(a+1)\frac{p-1}{p}}|\nabla \varphi |\right)d\mu
                   \nonumber\\
                   \leq & b\left(\int_M|\nabla u|^{p}\varphi^{b}u^{-a-1}d\mu \right)^{\frac{p-1}{p}}\cdot \left(\int_M \varphi^{b-p}u^{(a+1)(p-1)}|\nabla \varphi |^p d\mu \right)^{\frac{1}{p}}.
               \end{align}
               Similarly, we have
               \begin{align}
                   &b\int_M|\nabla u|^{q-2}\varphi^{b-1}\langle \nabla u,\nabla \varphi\rangle d\mu
                   \nonumber\\
                   \leq & b\left(\int_M|\nabla u|^{q}\varphi^{b}u^{-a-1}d\mu \right)^{\frac{q-1}{q}}\cdot \left(\int_M \varphi^{b-q}u^{(a+1)(q-1)}|\nabla \varphi |^q d\mu \right)^{\frac{1}{q}}.
               \end{align}
               From Lemma 2.2 and the definition of $Q$, we derive
               \begin{align}
                   \int_M |\nabla u|^pu^{-1-a}\varphi^b d\mu+\int_M |\nabla u|^q u^{-1-a}\varphi^b d\mu
                   \leq C a^{-1}Q.
               \end{align}
               From (2.11)-(2.14), we obtain
               \begin{align}
                   \int_MVu^s\varphi ^bd\mu\leq& C(a^{-1}Q)^{\frac{p-1}{p}}\cdot \left(\int_M \varphi^{b-p}u^{(a+1)(p-1)}|\nabla \varphi |^p d\mu \right)^{\frac{1}{p}}
                   \nonumber\\
                   &+C(a^{-1}Q)^{\frac{q-1}{q}}\cdot \left(\int_M \varphi^{b-q}u^{(a+1)(q-1)}|\nabla \varphi |^q d\mu \right)^{\frac{1}{q}}.
               \end{align}
               Applying H$\ddot{\mathrm{o}}$lder inequality again, we find that
               \begin{align}
                   &\int_M \varphi^{b-p}u^{(a+1)(p-1)}|\nabla \varphi |^p d\mu
                   \nonumber\\
                   =&\int_Mu^{(a+1)(p-1)}\varphi^{\frac{(a+1)(p-1)}{s}b}V^{\frac{(a+1)(p-1)}{s}} \cdot V^{-\frac{(a+1)(p-1)}{s}}\varphi^{b\frac{s-(p-1)(a+1)}{s}-p}|\nabla \varphi|^pd\mu
                   \nonumber\\
                   \leq& \left(\int_{M\setminus K}V^{-\frac{(a+1)(p-1)}{s-(p-1)(a+1)}}\varphi^{b-\frac{s}{s-(a+1)(p-1)}p}|\nabla \varphi|^{\frac{ps}{s-(a+1)(p-1)}}d\mu \right)^{\frac{s-(a+1)(p-1)}{s}}
                   \nonumber\\
                   &\cdot \left(\int_{M\setminus K}Vu^s\varphi^b d\mu\right)^{\frac{(a+1)(p-1)}{s}}.
               \end{align}
               Similarly, we have 
               \begin{align}
                   &\int_M \varphi^{b-q}u^{(a+1)(q-1)}|\nabla \varphi |^q d\mu
                   \nonumber\\
                   \leq& \left(\int_{M\setminus K}V^{-\frac{(a+1)(q-1)}{s-(q-1)(a+1)}}\varphi^{b-\frac{s}{s-(a+1)(q-1)}q}|\nabla \varphi|^{\frac{s}{s-(a+1)(q-1)}q}d\mu \right)^{\frac{s-(a+1)(q-1)}{s}}
                   \nonumber\\
                   &\cdot \left(\int_{M\setminus K}Vu^s\varphi^b d\mu\right)^{\frac{(a+1)(q-1)}{s}}.
               \end{align}
               Substituting (2.16) and (2.17) into (2.15), we obtain
               \begin{align*}
                   \int_MVu^s\varphi ^bd\mu\leq& C(a^{-1}Q)^{\frac{p-1}{p}}\cdot \left(\int_{M\setminus K}Vu^s\varphi^b d\mu\right)^{\frac{(a+1)(p-1)}{sp}}
                   \nonumber\\
                   &\cdot \left(\int_{M\setminus K}V^{-\frac{(a+1)(p-1)}{s-(p-1)(a+1)}}\varphi^{b-\frac{s}{s-(a+1)(p-1)}p}|\nabla \varphi|^{\frac{ps}{s-(a+1)(p-1)}}d\mu \right)^{\frac{s-(a+1)(p-1)}{sp}}
                   \\
                   &+C(a^{-1}Q)^{\frac{q-1}{q}}\cdot \left(\int_{M\setminus K}Vu^s\varphi^b d\mu\right)^{\frac{(a+1)(q-1)}{sq}}
                   \\
                   \cdot & \left(\int_{M\setminus K}V^{-\frac{(a+1)(q-1)}{s-(q-1)(a+1)}}\varphi^{b-\frac{s}{s-(a+1)(q-1)}q}|\nabla \varphi|^{ \frac{qs}{s-(a+1)(q-1)}}d\mu \right)^{\frac{s-(a+1)(q-1)}{sq}}
               \end{align*}
               Since $0\leq \varphi\leq 1$ and $b>\frac{2ps}{s-p+1}$, we obtain (2.10)
           \end{proof}
           At the end of the section, we give the following integral results under the condition HP1, HP2 and HP3.
           \begin{lemma}
                Let $f$ be a nonnegative decreasing function on $\mathbb{R_+} $, then for large enough $R$, we have
               \\
               (1) If HP1 holds, then 
               \begin{align}
                   \int_{M\setminus B_R} f(r(x))V^{-\bar{k}_z+\epsilon}d\mu \leq C\int_{\frac{R}{2}}^\infty f(r)r^{s_z+C_0\epsilon-1}(\mathrm{log}\ r)^{k_z}dr.
               \end{align}
                \\
               (2) If HP2 holds, then 
               \begin{align}
                   &\int_{M\setminus B_R} f(r(x))V^{-\bar{k}_z+\epsilon}d\mu \leq C\int_{\frac{R}{2}}^\infty f(r)r^{s_z+C_0\epsilon-1}(\mathrm{log}\ r)^{\bar{k}_z}dr
                   \\
                   &\int_{M\setminus B_R} f(r(x))V^{-\bar{k}_z-\epsilon}d\mu \leq C\int_{\frac{R}{2}}^\infty f(r)r^{s_z+C_0\epsilon-1}(\mathrm{log}\ r)^{\bar{k}_z}dr
               \end{align}
                \\
               (3) If HP3 holds, then 
               \begin{align}
                   \int_{M\setminus B_R} f(r(x))V^{-\bar{k}_z+\epsilon}d\mu \leq C\int_{\frac{R}{2}}^\infty f(r)r^{s_z+C_0\epsilon-1}(\mathrm{log}\ r)^{k}e^{-\epsilon\theta(\mathrm{log}r)^{\tau_z}}dr.
               \end{align}
           \end{lemma}
           \begin{proof}
               Since the proof is similar, we only prove (2.18). By the monotonicity of $f$, we have
               \begin{align*}
                   \int_{M\setminus B_R} f(r(x))V^{-\bar{k}_z+\epsilon}d\mu=&\sum\limits_{i=0}^\infty\int_{B_{2^{i+1}R}\setminus B_{2^{i}R}}f(r(x))V^{-\bar{k}_z+\epsilon}d\mu
                   \\
                   \leq & \sum\limits_{i=0}^\infty f(2^{i}R)\int_{B_{2^{i+1}R}\setminus B_{2^{i}R}}V^{-\bar{k}_z+\epsilon}d\mu
                   \\
                   \leq & C\sum\limits_{i=0}^\infty f(2^{i}R)(2^{i}R)^{s_z+C_0\epsilon}(\mathrm{log} (2^{i}R))^{k_z}
                   \\
                   \leq &C \sum\limits_{i=0}^\infty f(2^{i}R)(2^{i-1}R)^{s_z+C_0\epsilon-1}(\mathrm{log} (2^{i-1}R))^{k_z} \cdot (2^{i-1}R)
                   \\
                   \leq &C\int_{\frac{R}{2}}^\infty f(r)r^{s_z+C_0\epsilon-1}(\mathrm{log}R)^{k_z}dr.
               \end{align*}
           \end{proof}

           \section{Proof of Theorem 1.5}
           In this section, we prove the Theorem 1.5 by using Lemma 2.2.
           
           ~~\\
         $\mathit{Proof\ of\ Theorem\ 1.5.}$ For $R>0$ large enough, let $a=\frac{1}{\mathrm{log R}}$. Define $\varphi_n=\varphi(x)\eta_n(x)$ for $n\in \mathbb{N}$, where
         \begin{align*}
    \varphi(x) &= 
    \begin{cases}
        1, & \text{if }r(x)\leq R \\
        \left(\frac{r(x)}{R}\right)^{-C_1a},  & \text{if } r(x)>R
    \end{cases} \\
\end{align*}
    and 
         \begin{align*}
    \eta_n(x) &= 
    \begin{cases}
        1, & \text{if }r(x)< nR \\
        2-\frac{r(x)}{nR},  & \text{if } nR\leq r(x)< 2nR
        \\
        0,  & \text{if } r(x)\geq 2nR.
    \end{cases} \\
\end{align*}
Here we assume that $C_1>\frac{C_0+q+2}{qs}$ with $C_0$ as in HP1. Notice that 
\begin{align}
    |\nabla \varphi_n|^\alpha\leq C(|\nabla \varphi|^\alpha+\varphi^\alpha|\nabla\eta_n|^\alpha) 
\end{align}
for $\alpha\geq 1$. Now we use $\varphi_n$ as the test function in Lemma 2.2 and assume that $a,b,C$ are defined as in Lemma 2.2. Then we have
        \begin{align}
                  &\int_MVu^{s-a}\varphi_n^bd\mu 
                  \nonumber\\
                  \leq& C\left(a^{-\frac{s(p-1)}{s-p+1}}\int_M V^{-\frac{p-a-1}{s-p+1}} |\nabla \varphi_n|^{\frac{(s-a)p}{s-p+1}}d\mu+ a^{-\frac{s(q-1)}{s-q+1}}\int_M V^{-\frac{q-a-1}{s-q+1}} |\nabla \varphi_n|^{\frac{(s-a)q}{s-q+1}}d\mu \right)
                  \nonumber\\
                  \leq& Ca^{-\frac{s(p-1)}{s-p+1}}\left(\int_M V^{-\frac{p-a-1}{s-p+1}} |\nabla \varphi|^{\frac{(s-a)p}{s-p+1}}d\mu+\int_M V^{-\frac{p-a-1}{s-p+1}}\varphi^{\frac{p(s-a)}{s-p+1}} |\nabla \eta_n|^{\frac{(s-a)p}{s-p+1}}d\mu \right)
                  \nonumber\\
                  &+Ca^{-\frac{s(q-1)}{s-q+1}}\left(\int_M V^{-\frac{q-a-1}{s-q+1}} |\nabla \varphi|^{\frac{(s-a)q}{s-q+1}}d\mu+\int_M V^{-\frac{q-a-1}{s-q+1}}\varphi^{\frac{q(s-a)}{s-q+1}} |\nabla \eta_n|^{\frac{(s-a)q}{s-q+1}}d\mu \right)
                  \nonumber\\
                  =&Ca^{-\frac{s(p-1)}{s-p+1}}(I_{p,1}+I_{p,2})+Ca^{-\frac{s(q-1)}{s-q+1}}(I_{q,1}+I_{q,2}),
               \end{align}
               where 
               \begin{align*}
                   &I_{p,1}=\int_M V^{-\frac{p-a-1}{s-p+1}} |\nabla \varphi|^{\frac{(s-a)p}{s-p+1}}d\mu,\quad I_{p,2}=\int_M V^{-\frac{p-a-1}{s-p+1}}\varphi^{\frac{p(s-a)}{s-p+1}} |\nabla \eta_n|^{\frac{(s-a)p}{s-p+1}}d\mu,
                   \\
                   &I_{q,1}=\int_M V^{-\frac{q-a-1}{s-q+1}} |\nabla \varphi|^{\frac{(s-a)q}{s-q+1}}d\mu,\quad I_{q,2}=\int_M V^{-\frac{q-a-1}{s-q+1}}\varphi^{\frac{q(s-a)}{s-q+1}} |\nabla \eta_n|^{\frac{(s-a)q}{s-q+1}}d\mu.
               \end{align*}
               From (2.18) and the fact 
               \begin{align*}
                   R^a=R^{\frac{1}{\mathrm{log}R}}=e,\quad |\nabla\varphi|\leq CaR^{C_1a}r^{-C_1a-1},
               \end{align*}
               we obtain
               \begin{align*}
                   I_{p,1}\leq& \int_{M\setminus B_R} V^{-\frac{p-a-1}{s-p+1}} (CaR^{C_1a}r(x)^{-C_1a-1})^{\frac{(s-a)p}{s-p+1}}d\mu
                   \nonumber\\
                   \leq & Ca^{\frac{p(s-a)}{s-p+1}}\int_{M\setminus B_R} V^{-\frac{p-a-1}{s-p+1}} r(x)^{-(C_1a+1)\frac{(s-a)p}{s-p+1}}d\mu
                   \\
                   \leq & Ca^{\frac{p(s-a)}{s-p+1}}\int_{\frac{R}{2}}^\infty r^{ -(C_1a+1)\frac{p(s-a)}{s-p+1}+s_p+C_0\frac{a}{s-p+1}-1}(\mathrm{log}\ r)^{k_p}dr.
               \end{align*}
               Let $y=\alpha \mathrm{log}\ r$, where 
               \begin{align*}
                   \alpha=&(C_1a+1)\frac{p(s-a)}{s-p+1}-s_p-C_0\frac{a}{s-p+1}
                   \\
                   =&C_1a\frac{p(s-a)}{s-p+1}-\frac{a}{s-p+1}-C_0\frac{a}{s-p+1} .
               \end{align*}
               Since $ C_1>\frac{C_0+q+2}{qs}\geq \frac{C_0+p+2}{ps}$ and $a$ small enough, we have $\alpha \geq \frac{a}{s-p+1}>0$. Thus
               \begin{align*}
                   I_{p,1}\leq& Ca^{\frac{p(s-a)}{s-p+1}}\int_0^\infty e^{-y}y^{-k_p}\alpha^{-k_p-1}dy\leq Ca^{\frac{p(s-a)}{s-p+1}-k_p-1}.
               \end{align*}
               For $I_{p,2}$, by HP1 with $\epsilon=\frac{a}{s-p+1}$, we derive
               \begin{align*}
                   I_{p,2}\leq & \left(\sup_{B_{2nR}\setminus B_{nR}}\varphi\right)^{\frac{p(s-a)}{s-p+1}}\left(\frac{1}{nR}\right)^{\frac{p(s-a)}{s-p+1}}\int_{_{B_{2nR}\setminus B_{nR}}}V^{-\frac{p-1}{s-p+1}+\frac{a}{s-p+1}}d\mu 
                   \\
                   \leq & C n^{-C_1a\frac{p(s-a)}{s-p+1}}\left(\frac{1}{nR}\right)^{\frac{p(s-a)}{s-p+1}}(2nR)^{s_p+C_0\frac{a}{s-p+1} }(\mathrm{log}(2nR))^{k_p}
                   \\
                   \leq &C n^{-C_1a\frac{p(s-a)}{s-p+1}- \frac{p(s-a)}{s-p+1}+s_p+C_0\frac{a}{s-p+1}}(\mathrm{log}(2nR))^{k_p},
               \end{align*}
               where in the last equality follows from $R^a=e$. By the choice of $C_1$ and $a$ small enough, we have
               \begin{align*}
                   -C_1a\frac{p(s-a)}{s-p+1}- \frac{p(s-a)}{s-p+1}+s_p+C_0\frac{a}{s-p+1}\leq -\frac{a}{s-p+1}<0.
               \end{align*}
               Hence we have
               \begin{align*}
                   I_{p,2}\leq Cn^{-\frac{a}{s-p+1}}(\mathrm{log}(2nR))^{k_p}.
               \end{align*}
               Analogous to $I_{p,1}$ and $I_{p,2}$, we have
               \begin{align*}
                   I_{q,1}\leq Ca^{\frac{q(s-a)}{s-q+1}-k_q-1},\quad I_{q,2}\leq Cn^{-\frac{a}{s-q+1}}(\mathrm{log}(2nR))^{k_q}.
               \end{align*}
               From the estimates of $I_{p,1},I_{p,2}, I_{q,1},I_{q,2}$ and (3.2), we have
               \begin{align*}
                   \int_{B_R}Vu^{s-a}d\mu\leq& \int_MVu^{s-a}\varphi_n^bd\mu
                   \\
                   \leq & Ca^{-\frac{s(p-1)}{s-p+1}}(a^{\frac{p(s-a)}{s-p+1}-k_p-1}+n^{-\frac{a}{s-p+1}}(\mathrm{log}(2nR))^{k_p})
                   \\
                   &+Ca^{-\frac{s(q-1)}{s-q+1}}(a^{\frac{q(s-a)}{s-q+1}-k_q-1}+n^{-\frac{a}{s-q+1}}(\mathrm{log}(2nR))^{k_q}).
               \end{align*}
               Letting $n\xrightarrow{} \infty$, we obtain
               \begin{align*}
                   \int_{B_R}Vu^{s-a}d\mu \leq Ca^{\frac{p-1}{s-p+1}-k_p}+Ca^{\frac{q-1}{s-q+1}-k_q}.
               \end{align*}
            Note that $ k_z<\frac{z-1}{s-z+1}$ for $z\in \{p,q\}$. Letting $R\xrightarrow{}\infty$, then $a=\frac{1}{\mathrm{log}R}\xrightarrow{}0^+$. By Fatou's Lemma, we conclude that
            \begin{align*}
                \int_{M}Vu^{s}d\mu=0,
            \end{align*}
            which implies that $u=0$ a.e. on $M$.   {\qed}

            \section{Proof of Theorem 1.6}
            In this section, we assume that $C_1>\mathrm{max}\{\frac{C_0}{s-p+1},\frac{C_0+q+2}{qs}\} $. Since the proof relies on Lemma 2.3, we assume that $a,b$ satisfy the condition in Lemma 2.3.

          ~~\\
         $\mathit{Proof\ of\ Theorem\ 1.6.}$ For $R>0$ large enough, let $a=\frac{1}{\mathrm{log R}}$. We choose the same test function $\varphi_n=\varphi(x)\eta_n(x)$ as in Section 3. From Lemma 2.3, we have
         \begin{align}
                   \int_MVu^s\varphi_n^bd\mu
                   \leq& C(a^{-1}Q_n)^{\frac{p-1}{p}}\cdot \left(\int_{M\setminus K}Vu^s\varphi_n^bd\mu \right)^{\frac{(a+1)(p-1)}{sp}}
                   \nonumber\\
                   &\cdot\left(\int_{M\setminus K}V^{-\frac{(a+1)(p-1)}{s-(p-1)(a+1)}}|\nabla\varphi_n|^{\frac{ps}{s-(a+1)(p-1)}} d\mu\right)^{\frac{s-(a+1)(p-1)}{sp}}
                   \nonumber\\
                    &+C(a^{-1}Q_n)^{\frac{q-1}{q}}\cdot \left(\int_{M\setminus K}Vu^s\varphi_n^bd\mu \right)^{\frac{(a+1)(q-1)}{sq}}
                   \nonumber\\
                   &\cdot\left(\int_{M\setminus K}V^{-\frac{(a+1)(q-1)}{s-(q-1)(a+1)}}|\nabla\varphi_n|^{\frac{qs}{s-(a+1)(q-1)}} d\mu\right)^{\frac{s-(a+1)(q-1)}{sq}},
               \end{align}
               where $K=\{x\in M:\varphi_n(x)=1\}$ and 
               \begin{align*}
                   Q_n=a^{-\frac{s(p-1)}{s-p+1}}\int_M V^{-\frac{p-a-1}{s-p+1}} |\nabla \varphi_n|^{\frac{(s-a)p}{s-p+1}}d\mu+ a^{-\frac{s(q-1)}{s-q+1}}\int_M V^{-\frac{q-a-1}{s-q+1}} |\nabla \varphi_n|^{\frac{(s-a)q}{s-q+1}}d\mu.
               \end{align*}
               First, we estimate $Q_n$. Recalling the proof in Section 3, we actually estimate $Q_n$ under the condition HP1. Arguing as in Section 3 and using (2.19), with the only difference that $k_z$ is replaced by $\bar{k}_z$, we have
               \begin{align}
                   Q_n\leq& Ca^{-\frac{s(p-1)}{s-p+1}}(a^{\frac{p(s-a)}{s-p+1}-\bar{k}_p-1}+n^{-\frac{a}{s-p+1}}(\mathrm{log}(2nR))^{\bar{k}_p})
                   \nonumber\\
                   &+Ca^{-\frac{s(q-1)}{s-q+1}}(a^{\frac{q(s-a)}{s-q+1}-\bar{k}_q-1}+n^{-\frac{a}{s-q+1}}(\mathrm{log}(2nR))^{\bar{k}_q})
                   \nonumber\\
                   \leq & C(1+a^{-\frac{s(p-1)}{s-p+1}}n^{-\frac{a}{s-p+1}}(\mathrm{log}(2nR))^{\bar{k}_p}+a^{-\frac{s(q-1)}{s-q+1}}n^{-\frac{a}{s-q+1}}(\mathrm{log}(2nR))^{\bar{k}_q}).
               \end{align}
               Hence we only need to estimate 
               \begin{align*}
                   &J_p=\int_{M\setminus K}V^{-\frac{(a+1)(p-1)}{s-(p-1)(a+1)}}|\nabla\varphi_n|^{\frac{ps}{s-(a+1)(p-1)}} d\mu,
                   \\
                   &J_q=\int_{M\setminus K}V^{-\frac{(a+1)(q-1)}{s-(q-1)(a+1)}}|\nabla\varphi_n|^{\frac{qs}{s-(a+1)(q-1)}} d\mu.
               \end{align*}
               From (3.1), we have
               \begin{align*}
                   J_p\leq & C\int_M V^{-\frac{(a+1)(p-1)}{s-(p-1)(a+1)}}|\nabla\varphi|^{\frac{ps}{s-(a+1)(p-1)}} d\mu
                   \\
                   &+C\int_M V^{-\frac{(a+1)(p-1)}{s-(p-1)(a+1)}}\varphi^{\frac{ps}{s-(a+1)(p-1)}}|\nabla\eta_n|^{\frac{ps}{s-(a+1)(p-1)}} d\mu
                   \\
                   =&C(J_{p,1}+J_{p,2}).
               \end{align*}
               Define $\delta_p=\frac{(a+1)(p-1)}{s-(p-1)(a+1)}-\bar{k}_p$, then
               \begin{align*}
                   \frac{(p-1)sa}{(s-p+1)^2}\leq \delta_p=\frac{(p-1)sa}{(s-(a+1)(p-1))(s-p+1)}\leq \frac{2(p-1)sa}{(s-p+1)^2}
               \end{align*}
               for $a$ small enough and 
               \begin{align*}
                   \frac{ps}{s-(a+1)(p-1)}=s_p+p\delta_p.
               \end{align*}
               From (2.20), (3.1) and let $\epsilon=\delta_p$, we derive
               \begin{align*}
                   J_{p,1}=&\int_M V^{-\frac{(a+1)(p-1)}{s-(p-1)(a+1)}}|\nabla\varphi|^{\frac{ps}{s-(a+1)(p-1)}} d\mu
                   \\
                   \leq & \int_{M\setminus B_R}V^{-\bar{k}_p-\delta_p}(C_1aR^{C_1a}r(x)^{-C_1a-1})^{s_p+p\delta_p}d\mu
                   \\
                   \leq & C\int_{\frac{R}{2}}^\infty a^{s_p+p\delta_p }\cdot r^{-(C_1a+1)(s_p+p\delta_p)+s_p+C_0\delta_p-1}(\mathrm{log}\ r)^{\bar{k}_p}dr,
               \end{align*}
               where the last inequality follows from $R^a=e$. Let $y=\alpha \mathrm{log}\ r$, where
               \begin{align*}
                   -\alpha=-(C_1a+1)(s_p+p\delta_p)+s_p+C_0\delta_p.
               \end{align*}
               For small $a$, we have 
               \begin{align*}
                   -\alpha<-\frac{C_0as}{(s-p+1)^2}<0.
               \end{align*}
               Thus 
               \begin{align*}
                   J_{p,1}\leq Ca^{s_p+p\delta_p }\int_{\frac{R}{2}}^\infty e^{-y}y^{\bar{k}_p}\alpha^{-\bar{k}_p-1}dy\leq Ca^{s_p+p\delta_p-\bar{k}_p-1 }.
               \end{align*}
               In order to estimate $J_{p,2}$, we use (1.7) with $\epsilon=\delta_p$. Then we have
               \begin{align*}
                   J_{p,2}=&\int_M V^{-\frac{(a+1)(p-1)}{s-(p-1)(a+1)}}\varphi^{\frac{ps}{s-(a+1)(p-1)}}|\nabla\eta_n|^{\frac{ps}{s-(a+1)(p-1)}} d\mu
                   \\
                   \leq & \left(\sup\limits_{B_{2nR\setminus B_{nR}}}\varphi\right)^{s_p+p\delta_p}\left(\frac{1}{nR}\right)^{s_p+p\delta_p}\int_{B_{2nR\setminus B_{nR}} }V^{-\bar{k}_p-\delta_p}d\mu
                   \\
                   \leq &C n^{-C_1a(s_p+p\delta_p)}\left(\frac{1}{nR}\right)^{s_p+p\delta_p}(2nR)^{s_p+C_0\delta_p}(\mathrm{log}(2nR))^{\bar{k}_p}
                   \\
                   \leq & Cn^{ -(C_1a+1)(s_p+p\delta_p)+s_p+C_0\delta_p}(\mathrm{log}(2nR))^{\bar{k}_p},
               \end{align*}
               where the last equality follows from $R^a=e$. By the choice of $C_1$ and $a$ small enough, we have
               \begin{align*}
                   -(C_1a+1)(s_p+p\delta_p)+s_p+C_0\delta_p\leq -\frac{C_0sa}{(s-p+1)^2}<0.
               \end{align*}
               This yields that 
               \begin{align*}
                   J_{p,2}\leq Cn^{-\frac{C_0sa}{(s-p+1)^2} }(\mathrm{log}(2nR))^{\bar{k}_p}.
               \end{align*}
               From the estimate of $ J_{p,1}$ and $ J_{p,2}$, we obtain
               \begin{align}
                   J_p\leq C(a^{s_p+p\delta_p-\bar{k}_p-1 }+n^{-\frac{C_0sa}{(s-p+1)^2} }(\mathrm{log}(2nR))^{\bar{k}_p} ).
               \end{align}
               Similarly, we have
               \begin{align}
                   J_q\leq C(a^{s_q+q\delta_q-\bar{k}_q-1 }+n^{-\frac{C_0sa}{(s-q+1)^2} }(\mathrm{log}(2nR))^{\bar{k}_q} ),
               \end{align}
               where $\delta_q= \frac{(q-1)sa}{(s-(a+1)(q-1))(s-q+1)}$. Rewrite (4.1), we have
               \begin{align}
                   \int_MVu^s\varphi_n^bd\mu
                   \leq& C(a^{-1}Q_n)^{\frac{p-1}{p}}\cdot \left(\int_{M\setminus K}Vu^s\varphi_n^bd\mu \right)^{\frac{(a+1)(p-1)}{sp}}
                   \cdot J_p^{\frac{s-(a+1)(p-1)}{sp}}
                   \nonumber\\
                   &+C(a^{-1}Q_n)^{\frac{q-1}{q}}\cdot \left(\int_{M\setminus K}Vu^s\varphi_n^bd\mu \right)^{\frac{(a+1)(q-1)}{sq}}
                   \cdot J_q^{\frac{s-(a+1)(q-1)}{sq}}.
               \end{align}
               It is easy to see that there exists $\frac{(a+1)(p-1)}{sp}<\gamma<1$ such that
               \begin{align*}
                   1+\int_MVu^s\varphi_n^bd\mu
                   \leq& C[(a^{-1}Q_n)^{\frac{p-1}{p}}J_p^{\frac{s-(a+1)(p-1)}{sp}}+(a^{-1}Q_n)^{\frac{q-1}{q}}J_q^{\frac{s-(a+1)(q-1)}{sq}}+1]
                   \nonumber\\
                   &\cdot \left(1+\int_{M\setminus K}Vu^s\varphi_n^bd\mu\right)^\gamma.
               \end{align*}
               This implies that
               \begin{align}
                   &\left(1+\int_MVu^s\varphi_n^bd\mu\right)^{1-\gamma}
                   \nonumber\\
                   \leq& C[(a^{-1}Q_n)^{\frac{p-1}{p}}J_p^{\frac{s-(a+1)(p-1)}{sp}}+(a^{-1}Q_n)^{\frac{q-1}{q}}J_q^{\frac{s-(a+1)(q-1)}{sq}}+1].
               \end{align}
               Notice that, from (4.3), we have
               \begin{align*}
                   \limsup_{n\xrightarrow{}\infty} J_p\leq Ca^{s_p+p\delta_p-\bar{k}_p-1 }\leq C a^{\frac{(p-1)s}{s-(a+1)(p-1)}}.
               \end{align*}
               By a similar argument, we have
               \begin{align*}
                   \limsup_{n\xrightarrow{}\infty} Q_n \leq C, \quad  \limsup_{n\xrightarrow{}\infty} J_q\leq C a^{\frac{(q-1)s}{s-(a+1)(p-1)}}.
               \end{align*}
               This implies that 
               \begin{align*}
                   &\limsup_{n\xrightarrow{}\infty}[(a^{-1}Q_n)^{\frac{p-1}{p}}J_p^{\frac{s-(a+1)(p-1)}{sp}}+(a^{-1}Q_n)^{\frac{q-1}{q}}J_q^{\frac{s-(a+1)(q-1)}{sq}}+1]
                   \leq C.
               \end{align*}
                Hence, letting $n\xrightarrow{} \infty$ in (4.6), we obtain
               \begin{align*}
                   1+\int_MVu^s\varphi^bd\mu \leq C, 
               \end{align*}
               that is, 
               \begin{align*}
                   \int_MVu^s\varphi^bd\mu \leq C.
               \end{align*}
               Letting $n\xrightarrow{} \infty$ in (4.5), by the same argument, we have
               \begin{align*}
                   \int_{B_R}Vu^sd\mu
                   \leq& C \left(\int_{M\setminus B_R}Vu^s\varphi^bd\mu \right)^{\frac{(a+1)(p-1)}{sp}}+C \left(\int_{M\setminus B_R}Vu^s\varphi^bd\mu \right)^{\frac{(a+1)(q-1)}{sq}}.
               \end{align*}
               Letting $R\xrightarrow{}\infty$, we conclude 
               \begin{align*}
                   \int_{M}Vu^sd\mu=0,
               \end{align*}
               which implies that $u=0$ a.e. on $M$.   {\qed}

               \section{ Proof of Theorem 1.7}
~~\\
         $\mathit{Proof\ of\ Theorem\ 1.7.}$ For $R>0$ large enough, let $a=\frac{1}{\mathrm{log R}}$. Consider the same test function $\varphi_n$ as in Section 3 and assume that $C_1>\frac{C_0+q+2}{qs}$. Arguing as in the proof in Section 3, we have
         \begin{align}
                  &\int_MVu^{s-a}\varphi_n^bd\mu 
                  \nonumber\\
                  \leq& Ca^{-\frac{s(p-1)}{s-p+1}}\left(\int_M V^{-\frac{p-a-1}{s-p+1}} |\nabla \varphi|^{\frac{(s-a)p}{s-p+1}}d\mu+\int_M V^{-\frac{p-a-1}{s-p+1}}\varphi^{\frac{p(s-a)}{s-p+1}} |\nabla \eta_n|^{\frac{(s-a)p}{s-p+1}}d\mu \right)
                  \nonumber\\
                  &+Ca^{-\frac{s(q-1)}{s-q+1}}\left(\int_M V^{-\frac{q-a-1}{s-q+1}} |\nabla \varphi|^{\frac{(s-a)q}{s-q+1}}d\mu+\int_M V^{-\frac{q-a-1}{s-q+1}}\varphi^{\frac{q(s-a)}{s-q+1}} |\nabla \eta_n|^{\frac{(s-a)q}{s-q+1}}d\mu \right)
                  \nonumber\\
                  =&Ca^{-\frac{s(p-1)}{s-p+1}}(I_{p,1}+I_{p,2})+Ca^{-\frac{s(q-1)}{s-q+1}}(I_{q,1}+I_{q,2}).
               \end{align}
               Using (2.21) with $\epsilon=\frac{a}{s-p+1}$ and $R^a=e$, we have
               \begin{align*}
                   I_{p,1}\leq & C\int_{M\setminus B_R} V^{-\bar{k}_p+\epsilon}(ar(x)^{-C_1a-1})^{\frac{p(s-a)}{s-p+1}}d\mu\\
                   \leq & Ca^{\frac{p(s-a)}{s-p+1}}\int_{\frac{R}{2}}^\infty r^{-(C_1a+1)\frac{p(s-a)}{s-p+1}+s_p+C_0\frac{a}{s-p+1}-1}(\mathrm{log}\ r)^ke^{-\frac{a\theta}{s-p+1}(\mathrm{log}\ r)^{\tau_p}}dr. 
               \end{align*}
               By the choice of $C_1$, we have
               \begin{align*}
                   -(C_1a+1)\frac{p(s-a)}{s-p+1}+s_p+C_0\frac{a}{s-p+1}<0.
               \end{align*}
               Setting $ y=(\frac{a\theta}{s-p+1})^{\frac{1}{\tau_p}}\mathrm{log}\ r$, we have
               \begin{align*}
                   I_{p,1}\leq & C a^{\frac{p(s-a)}{s-p+1}}\int_0^\infty e^{-y^{\tau_p}}y^k\left(\frac{a\theta}{s-p+1}\right)^{-\frac{k+1}{\tau_p}}dy
                   \\
                   \leq &Ca^{\frac{p(s-a)}{s-p+1}-\frac{k+1}{\tau_p}}.
               \end{align*}
               As for $I_{p,2}$, by HP3 with $\epsilon=\frac{a}{s-p+1}$, we derive
               \begin{align*}
                   I_{p,2}\leq & \left(\sup_{B_{2nR}\setminus B_{nR}}\varphi\right)^{\frac{p(s-a)}{s-p+1}}\left(\frac{1}{nR}\right)^{\frac{p(s-a)}{s-p+1}}\int_{_{B_{2nR}\setminus B_{nR}}}V^{-\frac{p-1}{s-p+1}+\frac{a}{s-p+1}}d\mu 
                   \\
                   \leq & C n^{-C_1a\frac{p(s-a)}{s-p+1}}\left(\frac{1}{nR}\right)^{\frac{p(s-a)}{s-p+1}}(2nR)^{s_p+C_0\frac{a}{s-p+1} }(\mathrm{log}(2nR))^{k}e^{-\frac{a\theta}{s-p+1}(\mathrm{log}(2nR))^{\tau_p}}
                   \\
                   \leq &C n^{-C_1a\frac{p(s-a)}{s-p+1}- \frac{p(s-a)}{s-p+1}+s_p+C_0\frac{a}{s-p+1}}(\mathrm{log}(2nR))^{k}e^{-\frac{a\theta}{s-p+1}(\mathrm{log}(2nR))^{\tau_p}},
               \end{align*}
               where the last equality follows from $R^a=e$. Note that 
               \begin{align*}
                   &-C_1a\frac{p(s-a)}{s-p+1}- \frac{p(s-a)}{s-p+1}+s_p+C_0\frac{a}{s-p+1}
                   \\
                   =&\frac{a}{s-p+1}(-C_1p(s-a)+C_0+p)\leq -\frac{a}{s-p+1},
               \end{align*}
               since $C_1>\frac{C_0+p+2}{ps}$ and $a$ small enough. This yields that 
               \begin{align*}
                   I_{p,2}\leq C n^{-\frac{a}{s-p+1}}(\mathrm{log}(2nR))^{k}.
               \end{align*}
               Similarly, we have
               \begin{align*}
                   I_{q,1}\leq Ca^{\frac{q(s-a)}{s-q+1}-\frac{k+1}{\tau_q}}, \quad  I_{q,2}\leq C n^{-\frac{a}{s-q+1}}(\mathrm{log}(2nR))^{k}.
               \end{align*}
               Combining with (5.1), we find that
               \begin{align*}
                   &\int_MVu^{s-a}\varphi_n^bd\mu 
                  \nonumber\\
                  =&Ca^{-\frac{s(p-1)}{s-p+1}}(a^{\frac{p(s-a)}{s-p+1}-\frac{k+1}{\tau_p}}+n^{-\frac{a}{s-p+1}}(\mathrm{log}(2nR))^{k})
                  \\
                  &+Ca^{-\frac{s(q-1)}{s-q+1}}(a^{\frac{q(s-a)}{s-q+1}-\frac{k+1}{\tau_q}}+n^{-\frac{a}{s-q+1}}(\mathrm{log}(2nR))^{k}).
               \end{align*}
               Letting $n\xrightarrow{}\infty$ and noting that $a^a$ close to 1 for small $a$, we have
               \begin{align*}
                   \int_{B_R}Vu^{s-a}d\mu\leq& \int_MVu^{s-a}\varphi_n^bd\mu
                   \\
                   \leq & Ca^{\frac{s}{s-p+1}-\frac{k+1}{\tau_p}}+Ca^{\frac{s}{s-q+1}-\frac{k+1}{\tau_q}}.
               \end{align*}
               Notice that
               \begin{align*}
                   \frac{s}{s-p+1}-\frac{k+1}{\tau_p}>0,\quad \frac{s}{s-q+1}-\frac{k+1}{\tau_q}>0.
               \end{align*}
               Passing to the limit as $R$ tends to $\infty$, we conclude that
               \begin{align*}
                   \int_MVu^sd\mu=0,
               \end{align*}
               that is, $u=0$ a.e. on $M$.  {\qed}

        \section{An example}
        In this section, we will give an example to show the sharpness when $V(x)$ is a positive constant. For simplicity, we only prove the case $V(x)=1$. The idea comes from the work of \cite{SXX} on $m-$Laplace by gluing the function at infinity and near the origin. 
        \par 
        First, we will construct the manifold satisfying the volume condition (1.14). Let $(\mathbb{R}^n,g)$ be a Riemannian manifold with metric
        \begin{align}
            g=dr^2+\varphi^2(r)d\theta^2,
        \end{align}
        where $(r,\theta)$ are the polar coordinates and $\varphi(r)$ is a smooth, positive, increasing function such that
        \begin{align*}  
    \varphi(r)=\begin{cases}
        r, & \quad \forall  r<< 1,
           \\
           (r^{\alpha-1} (\mathrm{ln}\ r)^\beta)^{\frac{1}{n-1}},  &  \quad \forall r>>1.\\
    \end{cases} 
\end{align*}
       Here we choose $\alpha=\frac{sq}{s-q+1}>q,\beta>\frac{q-1}{s-q+1}$.  By the definition of $g$, it can be extended to the origin $O$, i.e. $(\mathbb{R}^n,g)$ is a complete manifold. It follows from (6.1) that
       \begin{align*}  
    S(r)=\begin{cases}
        \omega_n r^{n-1}, & \quad \forall  r<< 1
           \\
           \omega_n r^{\alpha-1} (\mathrm{ln}\ r)^\beta,  &  \quad \forall  r>>1,\\
    \end{cases} 
\end{align*}
          where $S(r)$ is the surface area of the ball $B_O(r)$ centered at origin $O$ and $ \omega_n$ is the surface area of the unit ball in $\mathbb{R}^n$. In particular, the monotonicity of $S$ follows from that of $\varphi$. Hence the volume of the ball $B_O(r)$ can satisfies
          \begin{align*}
              Vol(B_O(r))=\int_0^rS(t)dt\leq Cr^{\frac{sq}{s-q+1}} (\mathrm{ln}\ r)^\beta, \quad \forall  r>>1.
          \end{align*}
     Next, we construct the positive radial solution $u(x)=u(r)$ of the inequality
     \begin{align}
         \Delta_qu+ k\Delta_p u+\lambda u^s\leq 0,
     \end{align}
     where $ p> q>1,s>p-1$ and $0<k,\lambda\leq 1$ are constants to be determined later. Then the inequality (6.2) becomes
     \begin{align}
         (S|u^{'}|^{q-2}u^{'})^{'}+k(S|u^{'}|^{p-2}u^{'})^{'}+Su^s\leq 0.
     \end{align}
     The construction will be divided into two parts.

     ~\\
     $\mathbf{step\ 1.}$
     In this step, we construct $u$ at infinity and near the origin. For $r>R_0>>1$, we define 
     \begin{align*}
         u_a(r)=\int_r^{\infty}S(t)^{-\frac{1}{q-1}}(\mathrm{ln}\ t)^{\frac{a}{q-1}}dt,
     \end{align*}
     where $a>0,R_0$ are constants to be determined later. Thus,  
     \begin{align*}
         (S|u_a^{'}|^{q-2}u_a^{'})^{'}=&-((\mathrm{ln}\ r)^{a})^{'}=-a(\mathrm{ln}\ r)^{a-1}\frac{1}{r},
         \\
         (S|u_a^{'}|^{p-2}u_a^{'})^{'}=&-(S^{-\frac{p-q}{q-1}}(r)(\mathrm{ln}\ r)^{\frac{a(p-1)}{q-1}})^{'}=-\omega_n^{-\frac{p-q}{q-1}}(r^{-\frac{(p-q)(\alpha-1)}{q-1}}(\mathrm{ln}\ r)^{\frac{a(p-1)}{q-1}-\beta\frac{p-q}{q-1}})^{'}
         \\
         \leq &2\omega_n^{-\frac{p-q}{q-1}}\frac{(p-q)(\alpha-1)}{q-1}r^{-\frac{(p-q)(\alpha-1)}{q-1}-1}(\mathrm{ln}\ r)^{\frac{a(p-1)}{q-1}-\beta\frac{p-q}{q-1}}.
     \end{align*}
     Here we use $r>R_0>>1$ in the last equality. Similarly, we have
     \begin{align*}
         su_a^s\leq C Sr^sS^{-\frac{s}{q-1}}  (\mathrm{ln}\ r)^{\frac{as}{q-1}}=C\frac{1}{r}(\mathrm{ln}\ r)^{\frac{as-\beta(s-q+1)}{q-1}}.
     \end{align*}
     Noting that $0<k\leq 1$, we obtain 
     \begin{align*}
         &(S|u^{'}|^{q-2}u^{'})^{'}+k(S|u^{'}|^{p-2}u^{'})^{'}+Su^s
         \\
         \leq & -a(\mathrm{ln}\ r)^{a-1}\frac{1}{r}+2\omega_n^{-\frac{p-q}{q-1}}\frac{(p-q)(\alpha-1)}{q-1}r^{-\frac{(p-q)(\alpha-1)}{q-1}-1}(\mathrm{ln}\ r)^{\frac{a(p-1)}{q-1}-\beta\frac{p-q}{q-1}}
         \\
         &+C\frac{1}{r}(\mathrm{ln}\ r)^{\frac{as-\beta(s-q+1)}{q-1}}.
     \end{align*}
     Since $\beta>\frac{q-1}{s-q+1}$, there exists a constant $0<a<\beta-\frac{q-1}{s-q+1}$ such that
     \begin{align*}
         a-1>\frac{as-\beta(s-q+1)}{q-1}.
     \end{align*}
     Combining with $-1>-\frac{(p-q)(\alpha-1)}{q-1}-1 $, we find that 
     \begin{align}
         (S|u^{'}|^{q-2}u^{'})^{'}+k(S|u^{'}|^{p-2}u^{'})^{'}+Su^s\leq 0,\quad \forall r>>1.
     \end{align}
      For $R>R_0$, we define 
     \begin{align*}
         u_R(r)=\frac{R^{\gamma+1}-r^{\gamma+1}}{R^{\gamma+1}},\quad r\in [0,R],
     \end{align*}
     where $\gamma>0$ is a constant to be determined later. Using $S^{'}\geq 0$ and $0\leq u_R\leq 1$, we derive
     \begin{align*}
         (S|u_R^{'}|^{q-2}u_R^{'})^{'}=&-(S(\gamma+1)^{q-1}\frac{r^{(q-1)\gamma}}{R^{(\gamma+1)(q-1)}})^{'}
         \\
         \leq& -S(\gamma+1)^{q-1}(q-1)\gamma \frac{r^{(q-1)\gamma-1}}{R^{(\gamma+1)(q-1)}},
         \\
         (S|u_R^{'}|^{p-2}u_R^{'})^{'}\leq &-S(\gamma+1)^{p-1}(p-1)\frac{\gamma r^{(p-1)\gamma-1}}{R^{(\gamma+1)(q-1)}}\leq 0
     \end{align*}
     and $  Su_R^s\leq S.$ Letting $0<\gamma<\frac{1}{q-1}$ and $0<\lambda_0<(\gamma+1)^{q-1}(q-1)\gamma R^{-q}$, we have
     \begin{align}
         (S|u^{'}|^{q-2}u^{'})^{'}+k(S|u^{'}|^{p-2}u^{'})^{'}+\lambda_0 Su^s\leq 0.
     \end{align}
    
     ~\\
     $\mathbf{step\ 2.}$ In this step, we glue $u_a$ and $u_R$. From the definition of $u_a$ and $u_R$, we obtain
     \begin{align*}
         \frac{u_a^{'}(R_0)}{u_a(R_0)}<0,\quad \lim\limits_{R\xrightarrow{}\infty}\frac{u^{'}_R(R_0)}{u_R(R_0)}=0,\quad  \lim\limits_{R\xrightarrow{}R_0^{+}}\frac{u^{'}_R(R_0)}{u_R(R_0)}=-\infty.
     \end{align*}
     Due to the continuity of $\frac{u^{'}_R(R_0)}{u_R(R_0)} $, there exists a $R>R_0$ such that 
     \begin{align*}
         \frac{u_a^{'}(R_0)}{u_a(R_0)}=\frac{u^{'}_R(R_0)}{u_R(R_0)}.
     \end{align*}
     Hence we can define the function $u$ by 
     \begin{align*}  
    u(r)=\begin{cases}
        \frac{u_a(R_0)}{u_R(R_0)}u_R, & \quad   r\in [0,R_0],
           \\
           u_a ,  &  \quad   r\in (R_0,\infty).\\
    \end{cases} 
\end{align*}
      It is easy to see that $u$ is $C^1$. Moreover, from (6.4) and (6.5), we have 
      \begin{align}
         \Delta_qu+ k\Delta_p u+\lambda u^s\leq 0,
     \end{align}
    where $\lambda=min\{1,\lambda_0 \}$. Letting $k=\lambda^{\frac{p-q}{s-q+1}}$ and replacing $u$ by $\lambda^{-\frac{1}{s-q+1}}u$, we conclude that
     \begin{align}
         \Delta_qu+ \Delta_p u+u^s\leq 0.
     \end{align}

\bibliographystyle{siam}
\bibliography{ref}

~\\
  Biqiang Zhao
  \\
  $Beijing\ International\ Center\ for$
  \\
  $Mathematical\ Research $
\\
  $ Peking\ University$
\\
   $Beijing\ 100871 ,$ $P.R.\ China $

\end{document}